\documentclass[12pt,draftclsnofoot,onecolumn]{IEEEtran}

\usepackage{ifpdf}
\usepackage{graphicx,cite}
 
\usepackage{color}
\usepackage{graphicx}

\usepackage{amsmath,amsfonts,amssymb,amsthm,mathtools}
\usepackage{hyperref}
\usepackage{graphicx,color}

\newtheorem{proposition}{Proposition}[section]
\newtheorem{theorem}[proposition]{Theorem}

\newtheorem{lemma}[proposition]{Lemma}
\newtheorem{remark}[proposition]{Remark}

\DeclarePairedDelimiter\norm{\lVert}{\rVert}

\begin{document}
\title{\LARGE{Counting $k$-Hop Paths in the Random Connection Model}}

\author{Alexander~P.~Kartun-Giles,~\IEEEmembership{Member,~IEEE,}
        Sunwoo~Kim,~\IEEEmembership{Member,~IEEE}
\thanks{Alexander~P.~Kartun-Giles and Sunwoo Kim are with the Department of Electronic Engineering, Hanyang University, Seoul, South Korea.}}

\markboth{IEEE Transactions on Wireless Communications,~Vol.~xx, No.~xx, ~2017}%
{Jayaprakasam \MakeLowercase{\textit{et al.}}: Multi-objective Beampattern Optimization in Collaborative Beamforming via NSGA-II}

\maketitle
\begin{abstract}
We study, via combinatorial enumeration, the probability of $k$-hop connection between two nodes in a wireless multi-hop network. This addresses the difficulty of providing an exact formula for the scaling of hop counts with Euclidean distance without first making a sort of mean field approximation, which in this case assumes all nodes in the network have uncorrelated degrees. We therefore study the mean and variance of the number of $k$-hop paths between two vertices $x,y$ in the random connection model, which is a random geometric graph where nodes connect probabilistically rather than deterministically according to a critical connection range. In the example case where Rayleigh fading is modelled, the variance of the number of three hop paths is in fact composed of four separate decaying exponentials, one of which is the mean, which decays slowest as $\norm{x-y}\to\infty$. These terms each correspond to one of exactly four distinct sub-structures with can form when pairs of paths intersect in a specific way, for example at exactly one node. Using a sum of factorial moments, this relates to the path existence probability. We also discuss a potential application of our results in bounding the broadcast time.
\end{abstract}

\IEEEpeerreviewmaketitle

\section{Introduction}

We want to know what bounds we can put on the distribution of the graph distance, measured in hops, between two transceivers in a random geometric network, given we know their Euclidean separation. We might also ask, if we know the graph distance, what bounds can we put on the Euclidean distance? Routing tables already provide the shortest hop count to any other node, so this information alone is theoretically able to provide a quick, low power Euclidean distance estimate. These statistics can then be used to locate nodes via multilateration, by first placing anchors with known co-ordinates at central points in the network. They can also bound packet delivery delay, or bound the runtime of broadcast over an unknown topology \cite{mao2010,diaz2016}.

This question of low power localisation is important in the industrial application of wireless communications. But, in its own right, the question of relating graph to Euclidean distance is fundamental to the application of \textit{stochastic geometry} more generally, which is commonly used to calculate the macroscopic interference profile of e.g. an ultra-dense network, in terms of the microscopic details of the transmitter and receiver positions.

It is known that, in the connectivity regime of the random geometric graph, that as the system size goes to infinity, there exists a path of the shortest possible length between any pair of vertices with probability one. Relaxing the limit, and only requiring the graph to have a `giant' connected component, the graph distance is only ever a constant times the Euclidean distance (scaled by the connection range), i.e. that paths do not wander excessively from a theoretical straight line \cite{diaz2016}.

But the probability of $k$-hop connection, which is a natural question in mesh network theory, is not well understood. At least we can say, according to what has just been said, that for the smallest possible $k$, a path of that length will exist with probability one if the graph is connected with probability one. But since this regime requires expected vertex degrees to go to infinity, we naturally ask what can be said in sparser networks.

Routes into this problem are via mean field models. This abstracts the network as essentially non-spatial, wherein two nearby vertices have uncorrelated degree distributions. In fact, nearby vertices often share the same neighbours \cite{ta2007,mao2011}. Corrections to this can be made via \textit{combinatorial enumeration}, and we detail this novel idea in the rest of this article. We count the \textit{number} of paths of $k$-hops which join two distant vertices, and ask what proportion of the time this number is zero. Using the following relation between the factorial moments of  $\sigma_{k}(\norm{x-y})$, which is the number of $k$-hop paths between $x$ and $y$ at Euclidean distance $\norm{x-y}$ in some random geometric graph:
\begin{eqnarray}
 P(\sigma_{k}=t) = \frac{1}{t!}\sum_{i \geq 0} \frac{(-1)^{i}}{i!}\mathbb{E}\sigma_{k} (\sigma_{k}-1) \dots (\sigma_{k}-t-i+1),
\end{eqnarray}
we are able to deduce the probability of at least one $k$-hop path
\begin{eqnarray}
 P(\sigma_{k}>0) = 1 - \sum_{i \geq 0}\frac{(-1)^{i}}{i!}\mathbb{E}\left[\left(\sigma_{k}\right)_{i}\right]
\end{eqnarray}
where $\left(\sigma_{k}\right)_{i}$ is the descending factorial. The partial sums alternatively upper and lower bound this path-existence probability. So what are these factorial moments? In this article, we are able to deduce, in the case of the random connection model where Euclidean points are linked probabilistically according to some fading model, the first two. This requires closed forms for the mean and variance of $\sigma_{k}$ in terms of the connection function, and other system parameters such as node density. The remaining moments are then accessible via a recursion relation, which we intend to detail in a later work, but is currently out of reach.

This rest of this paper is structured as follows. In Section \ref{sec:rcm} we introduce the random connection model, and summarise our results concerning these moments. We then discuss the background to this problem, and related work in both engineering and applied mathematics. In two sections which follow, we then derive the mean and variance for the non-trivial case $k=3$, and the mean also for general $k \in \mathbb{N}$, showing how the variance is a sum of four terms. We also provide numerical corroboration, including of an approximation to the probability that there exist zero $k$-hop paths for each $k \in \mathbb{N}$ in terms of a sum of factorial moments. We finally conclude in Section \ref{sec:conclusion}.

\section{Summary of main results}\label{sec:rcm}
The \textit{Random Connection Model} is a graph $\mathcal{G}_{H}=(\mathcal{Y},E)$ formed on a random subset $\mathcal{Y}$ of $\mathbb{R}^{d}$ by adding an edge between distinct pairs of $\mathcal{Y}$ with probability $H(\norm{x-y})$, where $H: \mathbb{R}^{+} \to \left[0,1\right]$ is called the
\textit{connection function}, and $\norm{x-y}$ is Euclidean distance. Often $\mathcal{Y}$ is a Poisson point process of intensity $\rho \mathrm{d}x$, with $\mathrm{d}x$ Lesbegue measure on $\mathbb{R}^{d}$. By \textit{k-hop path} we mean a non-repeating sequence of $k$ adjacent edges joining two different vertices $x,y$ in the vertex set of $\mathcal{G}_{H}$. Since we only add edges between distinct pairs of $\mathcal{Y}$, vertices do not connect to themselves in what follows. This forbids paths of two hops becoming three hops simply by connecting vertices to themselves at some point along the path. See e.g. Fig. \ref{fig:allpaths}, which shows an example case for $k=3$. We also consider the practically important case of Rayleigh fading \cite{sklar1997} where, with $\beta > 0$ a parameter and $\eta > 0$ the path loss exponent, the connection function, with $\norm{x-y}>0$, is given by
\begin{eqnarray}\label{e:1}
H(\norm{x-y}) = \exp\left(-\beta\norm{x-y}^{\eta}\right)
\end{eqnarray}
and is otherwise zero. This choice is discussed in e.g. Section 2.3 of \cite{giles2016}. Note that we refer to \textit{nodes} when discussing actual communication devices in a wireless network, and \textit{vertices} when discussing their associated graphs directly.

Using the usual definition, a path, which consists of a sequence of \textit{hops}, is a \textit{trail} in which all vertices are distinct. Note that a trail is a walk in which all edges a distinct, and a walk is an alternating series of vertices and edges. So in these results we have no loops or self-loops, and all the paths are distinguishable from each other even if only on a single edge.

We now detail our main results.

\begin{theorem}[Expected number of $k$-hop paths for general $H$]\label{p:khopexpectationgeneral}
  Take a general connection function $H: \mathbb{R}^{+} \to [0,1]$. Define a new Poisson point process $\mathcal{Y}^{\star}$ which is $\mathcal{Y}$ conditioned on containing two specific points $x,y \in \mathbb{R}^{d}$ at Euclidean distance $\norm{x-y}$. Consider those two vertices $x,y$ in the vertex set of the random geometric graph $\mathcal{G}_{H}=(\mathcal{Y}^{\star},E)$, and set $x=z_0,y=z_k$. Then, in $\mathcal{G}_{H}$, the expected number of $k$-hop paths starting at $x$ and terminating at $y$ is
\begin{eqnarray}\label{e:khopexpectationgeneral}
\mathbb{E}\sigma_{k} = \rho^{k-1}\int_{\mathcal{V}^{dk-d}}\mathrm{d}z_{1} \dots\mathrm{d}z_{k-1}\prod_{i=0}^{k-1}H\left(\norm{z_{i}-z_{i+1}}\right).
\end{eqnarray}
\end{theorem}

\begin{proposition}[Expected number of $k$-hop paths for the specific case of Rayleigh fading]\label{p:kexpectation}
  Take $H$ from Eq. \ref{e:1}, and $d,\eta=2$ so that we consider free space propagation in two dimensions. Define a new Poisson point process $\mathcal{Y}^{\star}$ as in Theorem \ref{p:khopexpectationgeneral}, and the random geometric graph $\mathcal{G}_{H}=(\mathcal{Y}^{\star},E)$. Then, in $\mathcal{G}_{H}$, the expected number of $k$-hop paths starting at $x$ and terminating at $y$ is
\begin{eqnarray}\label{e:khopexpectation}
\mathbb{E}\sigma_{k} = \frac{1}{k}\left(\frac{\rho \pi}{\beta}\right)^{k-1}\exp\left(\frac{-\beta \norm{x-y}^{2}}{k}\right).
\end{eqnarray}
\end{proposition}

\begin{theorem}[Variance of the number of three-hop paths for the specific case of Rayleigh fading]\label{p:threehopvariance}
  Take $H$ from Eq. \ref{e:1}, and $d,\eta=2$ so that we consider free space propagation in two dimensions. Define a new Poisson point process $\mathcal{Y}^{\star}$ as in Theorem \ref{p:khopexpectationgeneral}, and the random geometric graph $\mathcal{G}_{H}=(\mathcal{Y}^{\star},E)$. Then, in this graph, the variance of the number of $k$-hop paths starting at $x$ and terminating at $y$ is
\begin{multline}\label{e:threevar}
  \mathrm{Var}\left(\sigma_{3}\right) = \mathbb{E}\sigma_{3} + \frac{\pi^3\rho^3}{\beta^3}\left(\frac{1}{4}\exp{\left( \frac{-\beta\norm{x-y}^2}{2}\right)}\right.\\\left. +\frac{1}{6}\exp{\left( \frac{-3\beta\norm{x-y}^2}{4}\right)}\right)
+\frac{\pi^2\rho^2}{8\beta^2}\exp{\left( -\beta\norm{x-y}^2\right)}.
\end{multline}
\end{theorem}

\begin{theorem}[Path existence probability for the specific case of Rayleigh fading]\label{t:moments}
 In $\mathcal{G}_{H}$, as discussed, the probability that a three-hop path exists between nodes at Euclidean separation $\norm{x-y}$ satisfies
\begin{eqnarray}
  P(\sigma_{3}>0) \geq 1 - \left[ 2\mathbb{E}\sigma_{3} - (\mathbb{E}\sigma_{3})^2 - \mathrm{Var}(\sigma_{3})  \right]
\end{eqnarray}
and these statistics are known in closed form according to our results above in terms of the point process density and model parameters, given the nodes are at a specific Euclidean separation.
\end{theorem}

\begin{figure}
  \noindent \begin{centering}
    \includegraphics[scale=0.62]{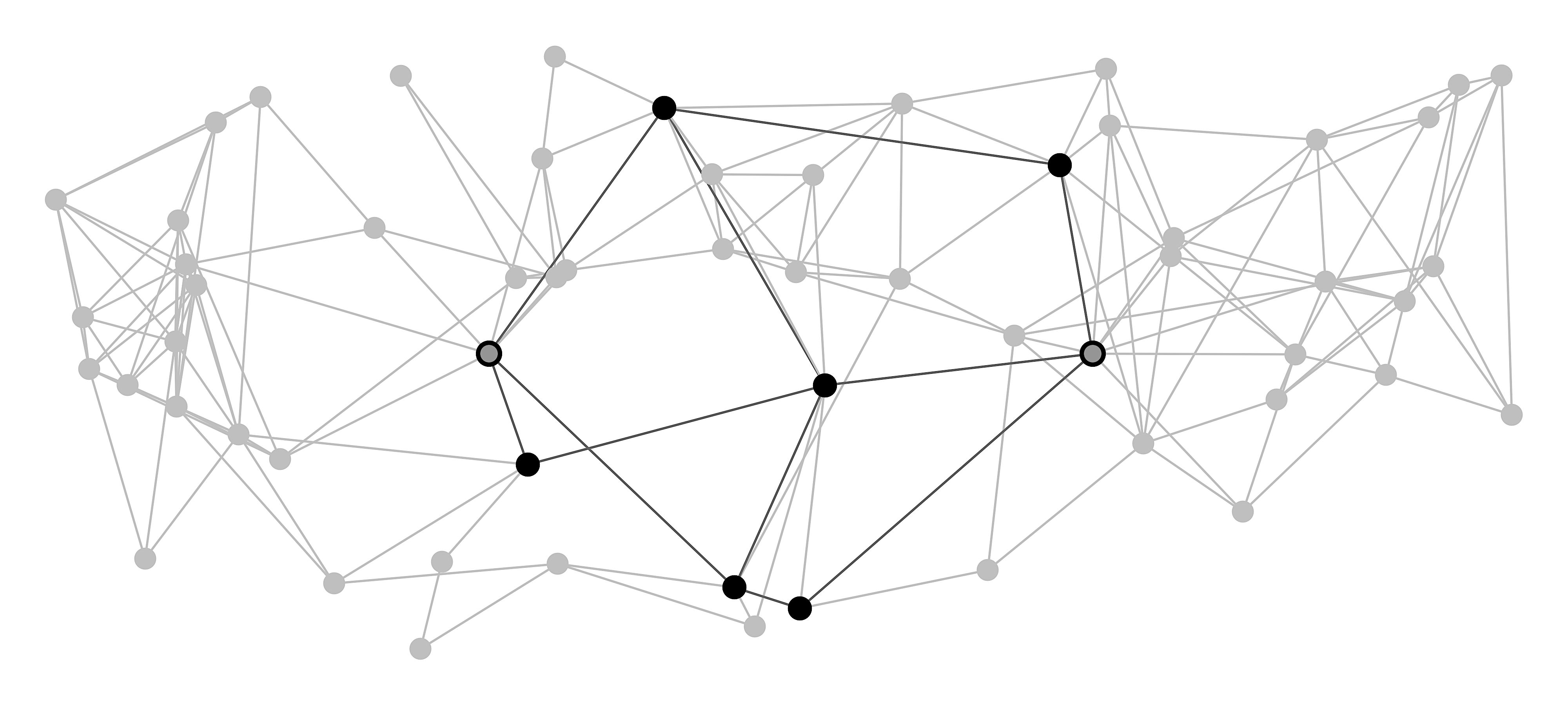}
\par\end{centering}
\caption{Example of the random connection model bounded inside a rectangle. All three hop paths between the two nodes with thick borders are highlighted. The Euclidean separation between $x$ and $y$ is 3 units taking $\beta=1$ and an expected $\rho=1$ node per unit area, $\sigma_{3}=5$, and $\mathbb{E}\sigma_{3}=2.36$ and $\mathrm{Var}(\sigma_{3})=9.95$, according to Eqs. \ref{e:khopexpectation} and \ref{e:threevar}.}
\label{fig:allpaths}
\end{figure}

\section{Background and Related Work}\label{sec:introduction}
We detail the historical background, and review what is known about this problem up to now. We also review the non-traditional random connection model, and highlight the applications of this theory in the developing field of mesh networks.

\subsection{Historical Background in Wireless Communications}
This problem has been around since the letter by S. A. G. Chandler in 1989 \cite{chandler1989}. Approximations are given to the probability that a path of $k$ hops will exist between $x$ and $y$ at $\norm{x-y}$ in a random geometric graph on a homogeneous Poisson point process. Similar attempts, concerning the slightly different problem of deducing the probability that two randomly selected stations are able to connect in one or two hops, with numerical corroboration, are presented in Bettstetter and Eberspacher  \cite{bettstetter2003}. More generally, given deterministic connection radius $r_{0}$, mean-field models are presented in Ta, Mao and Anderson \cite{ta2007}. They find a recursive formula for the probability of connection in $k$-hops or fewer. The contribution is of practical interest as an upper bound, though becomes inaccurate for large $k$ as the assumptions become unrealistic. It also weakens as the point process intensity gets large.

Later, Mao, Zhang and Anderson consider the problem in the case of probabilistic connection \cite{mao2010}. Again, an upper bound is presented, which is accurate for $k<3$. Again, for dense networks and/or large $k$, the bound is inaccurate. The concern, raised by the authors, is that a mean field approximation, which simply ignores the increased likelihood of a nodes sharing neighbours when they are close, does not yield much of an advance on this problem. It is beyond the scope of this article to quantify the effects this would bring about on network performance.

\subsection{Similar Advances in Applied Probability Theory}

From a pure mathematical viewpoint, most of the work related to this problem has been focused on studying upper bounds on the graph distance in terms of the Euclidean distance. As a key early result, Ellis, Martin and Yan showed that there exists some large constant $K$ such that for every $r_{0} \geq r_{c}$, and writing $d_{\text{Graph}}(x,y)$ for graph distance and $d_{\text{Euclidean}}(x,y)$ for Euclidean distance, for every pair of vertices
\begin{eqnarray}\label{e:linearbound}
  d_{Graph}(x,y) \leq K d_{Euclidean}(x,y) / r_{0}
\end{eqnarray}
This was extended by Bradonjic et al. for the supercritical percolation range of $r_{0}$, given $d_{Euclidean}(x,y)  = \Omega(\log^{7/2} n / r_{0}^2)$, i.e. given the Euclidean distance is sufficently large \cite{bradonjic2010}. Friedrich, Sauerwald and Stauffer improved this to $d_{Euclidean}(x,y)  = \Omega(\log n/r_{0})$. They also proved that if $r_{0}(n) = o(r_{c})$, with $r_{c}$ the critical radius for asymptotic connectivity with probability one, asymptotically almost surely there exist pairs of vertices with
$d_{Euclidean}(x,y) \leq 3r_{0}(n)$ and $d_{Graph}(x,y) = \Omega(\log n / r^{2})$, i.e. that the linear bound of Eq. \ref{e:linearbound} does not hold in the subconnectivity regime.

Most recently, the article of D\'{i}az, Mitsche, Perarnauand and P\'{e}rez-Gim\'{e}nez presents a rigorous proof of the fact that, in the connectivity regime, a path of the shortest possible length given the finite communication radius $r_{0}$ \cite{diaz2016} exists between a distant pair with probability one. This is equivalent to $K = 1 + o(1)$ asymptotically almost surely in Eq. \ref{e:linearbound}, given $r_{0} = \omega(r_{c})$.

\subsection{The Random Connection Model}
In the random subgraph of the complete graph on $n$ nodes obtained by including each of its edges independently with probability $p \sim \log n/n$, the probability that this graph is disconnected but free of isolated nodes tends to zero \cite{bollobas2001,penrose2013}. The \textit{random connection model} considers a random subgraph of the complete graph this time on a collection of nodes in a $d$-dimensional metric space \cite{iyer2015,mao2010,mao2011,mao2013}. The edge are added independently, but with probability $H: \mathbb{R} \to \left[0,1\right]$, where in the infinite space case $\int_{\mathbb{R}^{d}}H(x)\mathrm{d}x<\infty$ so that the expected vertex degrees are infinite only when the density is infinite. When the space is bounded, the graph is known as \textit{soft random geometric graph} \cite{penrose2016,cef2012,giles2016,penrosebook}.

In the language of theoretical probability theory, the connectivity threshold  \cite{penrose1997,mao2011,mao2013,iyer2015} goes as follows. Take a Poisson
point process $\mathcal{Y} \subset [0,1]^{d}$ of intensity $\lambda(n)\mathrm{d}x$, $\mathrm{d}x$ Lesbegue measure on $\mathbb{R}^{d} $, and $\left(\lambda\left(n\right)\right)_{n\in\mathbb{N}}$ an increasing $(0,\infty)$-valued sequence which goes to $\infty$ with $n$. Take the measurable function $H: \mathbb{R}^{+} \to [0,1]$ to be the probability that two nodes are joined by an edge. Then, as $\lambda\left(n\right) \to \infty$ along this sequence, in any limit where the expected number of isolated nodes converges to a positive constant $\alpha < \infty$, i.e.
\begin{equation}\label{e:degree}
\lambda \int_{\left[0,\sqrt{n}\right]^{2}}\exp\left(-\lambda\int_{\left[0,\sqrt{n}\right]^{2}} H\left(\norm{x-y}\right)\mathrm{d}y\right)\mathrm{d}x \to \alpha,
\end{equation}
their number converges to a Poisson distribution with mean $\alpha$, see Theorem 3.1 in Penrose's recent paper \cite{penrose2016}. The connection probability then follows, as before, from the probability that the graph is free of isolated vertices, given some conditions on the rate of growth of $H$ with $n$, see e.g. \cite{penrose2016,cef2012} for the case of random connection in a confined geometry, or non-convex geometry \cite{giles2016}, or for the random connection model \cite{mao2011}. Put simply, with any dense limit $\lambda(n) \to \infty$, the vertex degrees diverge, making isolation rare, and so vertices have degree zero  independently. They are therefore a homogeneous Poisson point process in space. Once this has occurred, all that is required is the expected number of isolated vertices $\alpha \to 0$. This is sufficient for connectivity, because large clusters merge, which is the main part of the proof.

The recent consensus is that spatial dependence between the node degrees, which form a Markov random field \cite{clifford1990} in the deterministic, finite range case, appears to preclude exact description of a map between distances, given by some norm, and the space of distributions describing the probability of $k$-hop connection between two nodes of known displacement, see e.g. Section 1 of \cite{ta2007}. We believe, however, that the number $\sigma_{k}(\norm{x-y})$ of $k$-hop paths may have a probability generating function similar to a $q$-series common in other combinatorial enumeration problems \cite{goulden2004}. In fact, the first author has demonstrated that this is indeed the case under deterministic connection in one dimension, proving that $\mathbb{E}q^{\sigma_{k}}$ is a random $q$-multinomial coefficient \cite{giles2017}. This is also studied in a single dimension recently in vehicular networks \cite{knight2017}.

\subsection{Impacts in Wireless Communications}
Bounds on the distribution of the number of hops between two points in space, for example, has been a recent focus of many researchers interested in the statistics of the number of hops to e.g. a sink in a wireless sensor network, or gateway-enabled small cell in an ultra-dense deployment of non-enabled smaller cells \cite{dulman2006,mao2011,ta2007,ge2016}, since it relates to data capacity in e.g. multihop communication with infrastructure support \cite{ng2010,zemlianov2005}, route discovery \cite{bettstetter2003,perevalov2005} and localisation \cite{nguyen2015}. We now detail other important examples of ongoing research where new insight on this problem will prove useful.

Firstly, consider the problem of broadcast. Broadcasting information from one node to eventually all other nodes in a network is a classic problem at the interface of applied mathematics and wireless communications. The task is to take a message available at one node, and by passing it away from that node, and then from its neighbours, make the message available to all the nodes in either the least time, or using the least energy, or using the fewest number of transmissions. If the nodes form the state space of a Markov chain, with links weighted with transmission probabilities, rapid mixing of this chain implies fast broadcast \cite{elsasser2007}.

When the network considered is random and embedded in space, the problem is this: given a graph $\mathcal{G}_{H}$, a source node is selected. This source node has a message to be delivered to all the other nodes in the network. Nodes not within range of the source must receive the message indirectly, via multiple hops. The model most often considered in the literature on broadcasting algorithms is synchronous. All nodes have clocks whose ticks measure time steps, or \textit{rounds}. A broadcasting operation is a sequence $(T)_{i \leq T}$, where each element is a set of nodes that act as transmitters in round $i$. The execution time is the number of rounds required before all nodes hear the message, which is the length $T$ of the broadcast sequence. See the review of Peleg for the case where collision avoidance is also considered \cite{peleg2007}.

Let $T_{p}$ be the number of rounds required for the message to be broadcast to all nodes with probability $p$. Then, since every broadcasting algorithm requires at least $\max \{ \log_{2} n, \text{diam}\left(\mathcal{G}\right) \}$, see \cite[Section 2]{elsasser2007}, the algorithm is called \textit{asymptotically optimal} if $T_{1-1/N} = \mathcal{O}\left(\log n + \text{diam}\left(\mathcal{G}\right)\right)$. This relates the broadcast time on \textit{random} graph, measured in rounds, to its diameter, i.e. the length in hops of longest path between any pair of nodes, given the path is geodesic. The running times are in fact often bounded in other ways by the diameter. The diameter for sufficiently distant nodes in e.g. connected graphs, is never more than a constant times the Euclidean distance, and for sufficiently dense graphs, is precisely the ceiling of the Euclidean distance \cite{diaz2016}. For more general limits, however, the diameter remains important, and so by providing a relation between graph distance and Euclidean distance, in a finite domain the broadcast time can be adequately bounded by our results.

Also, consider distance estimation. Hop counts between nodes in geometric networks can give a quick and energy efficient estimate of Euclidean distance. These estimates are made more accurate once this relation is well determined. This is often done numerically \cite{nguyen2015}. Given the graph distance, demonstrating how bounds on the Euclidean distance can work to lower power consumption in sensor network localisation is an important open problem. Using the multiplicity of paths can also assist inter-point distance estimates, though this research is ongoing, since one needs a path statistic with exceptionally low variance for this task to avoid errors. Either way, knowing the length in hops of geodesics given the Euclidean separation is essential.

Finally, we highlight packet delivery delay statistics and density estimation. Knowing the delay in packet transfer over a single hop is a classic task in ad hoc networks. With knowledge of the total number of hops, such as its expectation, one can provide the statistics of packet delivery delays given the relation of graph to Euclidean distance. This can help bound delivery delay in e.g. the growing field of delay tolerant networking, by adjusting network parameters accordingly.

\section{The expected number of $k$-hop paths}\label{sec:exp}
Consider a point process $\mathcal{X}$ on some space $\mathcal{V}$. If it is assumed that $x \in \mathcal{V}$ and $x \in \mathcal{X}$, what is true of the remaining points $\mathcal{X} \setminus \{x\}$? The Poisson point process has the property that when fixing a point, the remaining points are still a point process, and of the original intensity. This is \textit{Slivnyak's theorem}, and it characterises the Poisson process, see e.g. Proposition 5 in \cite{jagers1973}.  The relevance of the following lemma \cite{penrose2016,penrose20162} from stochastic geometry \cite{haenggi2009,elsawy2017} is now framed.
\begin{lemma}[Slivnyak-Mecke Formula]\label{l:mecke}
Let $t \in \mathbb{N}$. For any measurable real valued function $f$ defined on the product of $(\mathbb{R}^{d})^{t} \times \mathcal{G}$, where $\mathcal{G}$ is the space of all graphs on finite subsets of $[0,1]^{d}$, given a connection function $H$, the following relation holds
\begin{multline}\label{e:meckeformula}
\mathbb{E}\sum_{X_{1},\dots,X_{t} \in \mathcal{Y}}^{\neq}f\left(X_{1},\dots,X_{t},\mathcal{G}_{H}\left(\mathcal{Y}\setminus \{X_{1},\dots,X_{t}\}\right)\right) \\ = n^{t}\int_{\left[0,1\right]^{d}}\mathrm{d}x_{1}\dots\int_{\left[0,1\right]^{d}}\mathrm{d}x_{t}\mathbb{E}f\left(x_1,\dots,x_{t},\mathcal{G}_{H}\left(\mathcal{Y}\right)\right)
\end{multline}
where $\mathcal{Y}\subset[0,1]^{d}$, $\mathbb{E}\norm{\mathcal{Y}}=n$, and $\sum^{\neq}$ means the sum over all \textit{ordered} $t$-tuples of distinct points in $\mathcal{Y}$.
\end{lemma}
\begin{remark}
To clarify, note that $\{a,b\}$ and $\{b,a\}$ are distinct \textit{ordered} 2-tuples, but indistinct \textit{unordered} 2-tuples.
\end{remark}
\begin{proof}
In the case $t=2$ with
\begin{equation}\label{e:indicator1}
f\left(u,v,\mathcal{G}_{H}\left(\mathcal{Y}\right)\right) =: \mathbf{1}\{u \leftrightarrow v\}
\end{equation}
a Bernoulli variate with parameter $H(\norm{u-v})$, then
  \begin{equation}\label{e:meckeintegrand1}
    \mathbb{E}f\left(u,v,\mathcal{G}_{H}\left(\mathcal{Y}\right)\right) = H\left(\norm{u-v}\right)
  \end{equation}
  where the expectation is over all graphs $\mathcal{G}_{H}\left(\mathcal{Y}\right)$. These indicator functions are important for dealing with the existence of edges between points of $\mathcal{Y}$. We note this for clarity, but it is not required for what follows. The proof of Lemma \ref{l:mecke} is obtained by conditioning on the number of points of $\mathcal{Y}$. Firstly,
\begin{multline}
\mathbb{E}\sum_{X_{1},\dots,X_{m} \in \mathcal{Y}}^{\neq}f\left(X_{1},\dots,X_{m},\mathcal{G}_{H}\left(\mathcal{Y}\setminus \{X_{1},\dots,X_{m}\}\right)\right) \nonumber \\ 
=\sum_{t=m}^{\infty}\left(\frac{e^{n}n^{t}}{t!}\right)\left(t\right)_{m}\int_{[0,1]^{d}}\mathrm{d}x_{1}\dots \\ \dots \int_{[0,1]^{d}}\mathrm{d}x_{t}f\left(x_1 \dots x_{m},\mathcal{G}_{H}\left(\{x_{m+1}, \dots , x_{t}\}\right)\right)
\end{multline}
where $(n)_{k}=n(n-1)\dots(n-k-1)$ is the descending factorial. Bring the $m$-dimensional integral over positions of vertices in the $m$-tuple outside the sum,
\begin{multline}
n^{m}\int_{[0,1]^{d}}\mathrm{d}x_{1} \dots \int_{[0,1]^{d}}\mathrm{d}x_{m} \sum_{t=m}^{\infty}\left(\frac{e^{n}n^{t-m}}{(t-m)!}\right) \\ \times \int_{[0,1]^{d}}\mathrm{d}y_{1}\dots \\ \dots\int_{[0,1]^{d}}\mathrm{d}y_{t-m}f\left(x_1 \dots x_{m},\mathcal{G}_{H}\left(\{y_{1}, \dots , y_{t-m}\}\right)\right), \nonumber
\end{multline}
and change variables such that $r=t-m$, such that
\begin{multline}
n^{m}\int_{[0,1]^{d}}\mathrm{d}x_{1} \dots \int_{[0,1]^{d}}\mathrm{d}x_{m} \sum_{r=0}^{\infty}\left(\frac{e^{n}n^{r}}{r!}\right)\int_{[0,1]^{d}}\mathrm{d}y_{1}\dots \\ \dots \int_{[0,1]^{d}}\mathrm{d}y_{r}f\left(x_1 \dots x_{m},\mathcal{G}_{H}\left(\{y_{1}, \dots , y_{r}\}\right)\right). \\ = n^{m}\int_{\left[0,1\right]^{d}}\mathrm{d}x_{1}\dots\int_{\left[0,1\right]^{d}}\mathrm{d}x_{m}\mathbb{E}f\left(x_1,\dots,x_{m},\mathcal{G}_{H}\left(\mathcal{Y}\right)\right) \nonumber
\end{multline}
as required.
\end{proof}
We know provide a general formula for the expected number of $k$-hop paths between $x,y \in V$.
\begin{proof}[Proof of Theorem \ref{p:khopexpectationgeneral}]
Define a new Poisson point process $\mathcal{Y}^{\star}$ conditioned on containing two specific points $x,y \in \mathbb{R}^{d}$ at Euclidean distance $\norm{x-y}$ and set $x=z_0,y=z_k$. In a similar manner to Eq. \ref{e:indicator1}, define the \textit{path-existence function} $g$ to be the following product
\begin{equation}\label{e:g}
g\left(z_1,\dots,z_{k-1},\mathcal{G}_{H}\left(\mathcal{Y}^{\star}\right)\right) = \prod_{i=0}^{k-1}\mathbf{1}\{z_{i} \leftrightarrow z_{i+1} \} 
\end{equation}
where the indicator is defined in Eq. \ref{e:indicator1}. The expected value of this function is then just the product of the connection probabilities $H$ of the inter-point distance along the sequence $z_0,\dots,z_{k}$, i.e.
\begin{equation}\label{e:pathfunction}
\mathbb{E}g(z_1,\dots,z_{k-1},\mathcal{G}_{H}\left(\mathcal{Y}^{\star}\right)) =  \prod_{i=0}^{k-1}H\left(\norm{z_{i}-z_{i+1}}\right)
\end{equation}
From the Mecke formula
\begin{multline}\label{e:khopmecke}
\mathbb{E}\sum_{X_1,\dots,X_{k-1} \in \mathcal{Y}^{\star}}^{\neq} g\left(X_1,\dots,X_{k-1},\mathcal{G}_{H}\left(\mathcal{Y}\setminus \{X_{1},\dots,X_{k-1}\}\right)\right) \\ = \rho^{k-1}\int_{\mathbb{R}^{dk-d}} \mathbb{E}g\left(z_1,\dots,z_{k-1}\right)\mathrm{d}z_1 \dots \mathrm{d}z_{k-1}
\end{multline}
and with Eq. \ref{e:pathfunction} replacing the integrand on the right hand side, the proposition follows.
\end{proof}
We now expand on the practically important situation where vertices connect with probability given by Eq. \ref{e:1}.

\begin{proof}[Proof of Proposition \ref{p:kexpectation}]
Using Eq. \ref{e:khopmecke} with $H$ taken from Eq. \ref{e:1}, and in the case $d,\eta=2$, we have
\begin{multline}\mathbb{E}\sum_{X_1,\dots,X_{k-1} \in \mathcal{Y}^{\star}}^{\neq} g\left(X_1,\dots,X_{k-1},\mathcal{G}_{H}\left(\mathcal{Y}^{\star} \setminus X_1,\dots,X_{k-1}\right)\right)\\=
\rho^{k-1}\int_{-\infty}^{\infty}\dots\int_{-\infty}^{\infty} \mathrm{d}z_{1_{x}}\mathrm{d}z_{1_{y}} \dots \mathrm{d}z_{\left(k-1\right)_{x}}\mathrm{d}z_{\left(k-1\right)_{y}} \\ \times \exp\left({-\beta\left(z_{1_{x}}^{2}+\dots+\left(\norm{x-y}-z_{\left(k-1\right)_{x}}^{2}\right)+z_{\left(k-1\right)_{y}}^{2}\right)}\right) \nonumber
\end{multline}
which, due to the addition of terms in the exponent, factors into a product of integrals, to be performed in sequence. Each is in $d=2$ variables.

For example, in the case of $k=2$ hops, we have
\begin{multline}
\mathbb{E}\sum_{X_1 \in \mathcal{Y}^{\star}} g\left(X_1,\mathcal{G}_{H}\left(\mathcal{Y}^{\star} \setminus X_1\right)\right) = \rho \int_{-\infty}^{\infty} \mathrm{d}z_{1_{x}} \int_{-\infty}^{\infty} \mathrm{d}z_{1_{y}} \dots \\  \times \exp{\left(-\beta \left( z_{1_{x}}^{2} +  z_{1_{y}}^{2} +  \left(||x-y|| - z_{1_{x}}\right)^{2} +  z_{1_{y}}^{2} \right) \right)} \nonumber
\end{multline}
which, by expanding the exponent becomes
\begin{eqnarray}
&& \rho \int_{-\infty}^{\infty}\exp{\left(-2\beta z_{1_{y}}^{2} \right)}\mathrm{d}z_{1_{y}} \int_{-\infty}^{\infty} \exp{\left(-\beta \left( z_{1_{x}}^{2} + \left(||x-y|| - z_{1_{x}}\right)^{2} \right) \right)}\mathrm{d}z_{1_{x}} \nonumber\\  &=& \frac{ \rho \pi}{8 \beta} \exp{\left(\frac{-\beta \norm{x-y}^2}{2} \right)}\left[\text{Erf}\left(\frac{\left(2z_{1_{x}}-\norm{x-y}\right)\sqrt{\beta}}{\sqrt{2}}\right)\right]^{\infty}_{-\infty} \left[\text{Erf}\left(\sqrt{2 \beta} z_{1_{y}}\right)\right]^{\infty}_{-\infty}\nonumber \\ &=&  \frac{\rho\pi}{2 \beta} \exp{\left(\frac{-\beta \norm{x-y}^2}{2} \right)} \nonumber
\end{eqnarray}
where the second integral on the right hand side of the second line can be performed by completing the square, and then substituting to obtain a Gaussian integral (which integrates to an error function). Due to the limits, the integrals evalaute in turn exactly for each $k$. Comparing these results, determined one by one, demonstrates the general form of Eq. \ref{e:khopexpectation}.
\end{proof}

\section{The variance for $k=3$}\label{sec:var}
In this section we consider the variance of the number of paths of three sequential edges.
\begin{proof}[Proof of Theorem \ref{p:threehopvariance}]
  A similar technique to the one implemented here is used to derive the asymptotic variance of the number of edges in the random geometric graph $G$ defined in Section \ref{sec:rcm}, see e.g. Section 2 of \cite{penrose20162}.

The proof now follows. Consider $\mathbb{E}\sigma_{3}^{2}\left(\norm{x-y}\right)$. This is the expected number of ordered \textit{pairs} of three hop paths between the fixed vertices $x$ and $y$. There are three non-overlapping contributions,
\begin{eqnarray}\label{e:threeterms}
\sigma_{3}^{2} = \Sigma_{0} + \Sigma_{1} + \Sigma_{2},
\end{eqnarray}
where for $i=0,1,2$ the integer $\Sigma_{i}$ denotes the number of ordered pairs of three hop paths with $i$ vertices in common. Taking $g$ from Eq. \ref{e:g}, we can quickly evaluate the term $\Sigma_{0}$, which is the following sum over ordered 4-tuples of points in $\mathcal{Y}^{\star}$,
\begin{eqnarray}
\Sigma_{0} = \sum_{V,W,X,Y \in \mathcal{Y}^{\star}}^{\neq}g\left(V,W\right)g\left(X,Y\right).
\end{eqnarray}
The Mecke formula implies that
\begin{eqnarray}\label{e:sigmazero}
  \mathbb{E}\Sigma_{0} = \rho^{4}\int_{\mathbb{R}^{8}}\mathbb{E}\left(g\left(z_1,z_2\right)g\left(z_3,z_4\right)\right)\mathrm{d}z_1\mathrm{d}z_2\mathrm{d}z_3\mathrm{d}z_4,
\end{eqnarray}
and since, according to Eq. \ref{e:khopmecke}, we have
\begin{eqnarray}
  \mathbb{E}\sigma_{3}=\rho^{2}\int_{\mathbb{R}^{4}}\mathbb{E}g\left(z_1,z_2\right)\mathrm{d}z_1\mathrm{d}z_2
\end{eqnarray}
then $\mathbb{E}\Sigma_{0}=\left(\mathbb{E}\sigma_{3}\right)^2$, which cancels with a term in the definition of the variance $\mathrm{Var}(\sigma_{3}) = \mathbb{E}(\sigma_{3}^2)-\left(\mathbb{E}(\sigma_{3})\right)^{2}$, such that we have the following simpler expression for the variance, based on Eq. \ref{e:threeterms} and Eq. \ref{e:sigmazero},
\begin{eqnarray}\label{e:vardecomposition}
\mathrm{Var}(\sigma_{3})= \mathbb{E}\Sigma_1 + \mathbb{E}\Sigma_2.
\end{eqnarray}
Now, $\mathrm{Var}(\sigma_{3})$ will follow from a careful evaluation of $\Sigma_1$ and $\Sigma_2$. The first of these, $\Sigma_{1}$, can be broken down into two separate contributions, denoted $\Sigma_{1(1)}$ and $\Sigma_{1(2)}$.

Dealing first with $\Sigma_{1(1)}$, notice the left panel of Fig. \ref{fig:paths1}, which shows an intersecting pair of paths in $\mathcal{G}_{H}\left(\mathcal{Y}^{\star}\right)$ which share a single vertex $U$ which is itself connected by an edge to $y$. Many triples of points in $\mathcal{Y}^{\star} \setminus \{x,y\}$ display this property.

We want a sum of indictor functions which counts the number of pairs of three-hop paths which intersect at a single vertex, in precisely the manner of the left panel of Fig. \ref{fig:paths1}. In the following double sum, the function $g(A,B)$ indicates that vertices $A$ and $B$ are on a three hop path $x \leftrightarrow A \leftrightarrow B \leftrightarrow y$. Look with care at the limits of the sum. They are set up to only count pairs of paths (2-tuples) which intersect this very specific way:
\begin{eqnarray}\label{e:s1}
\Sigma_{1(1)} = \sum_{U \in \mathcal{Y}^{\star}} \sum_{W,Z \in \mathcal{Y}^{\star} \setminus \{U\} }^{\neq}g\left(U,W\right)g\left(U,Z\right),
\end{eqnarray}
Via the Mecke formula, and with $U$ the position vector of the shared vertex, this can be written as an integral:
\begin{multline}\label{e:s1int}
  \rho\int_{\mathbb{R}^{d}}H\left(\norm{x-U}\right) \mathbb{E}\left[\left(\sum_{X \in \mathcal{Y}^{\star} \setminus \{U\} }\mathbf{1}\{U \leftrightarrow X\}\mathbf{1}\{X \leftrightarrow y\}\right)_{2}\right]\mathrm{d}U \\ +  \rho\int_{\mathbb{R}^{d}}H\left(\norm{y-U}\right) \mathbb{E}\left[\left(\sum_{X \in \mathcal{Y}^{\star} \setminus \{U\} }\mathbf{1}\{U \leftrightarrow X\}\mathbf{1}\{X \leftrightarrow x\}\right)_{2}\right]\mathrm{d}U
\end{multline}
with $(a)_{2}=a(a-1)$. This descending factorial counts distinct pairs of a set with cardinality $a$. If there are $9$ two hops paths, there are $9(8)=72$ pairs of paths which are not paths paired with themselves. Also, there are two terms in Eq. \ref{e:s1int} because we can exchange $x$ and $y$, and get another structure on a single triple of point which must be counted as part of $\Sigma_{1}$. For the first term, three hop paths diverge initially from each other, then unite at $U$ to hop in unison to $y$. The others hop first in union to $U$, then diverge to meet again finally at $y$.
\begin{figure*}[t]
  \noindent \begin{centering}
    \includegraphics[scale=0.35]{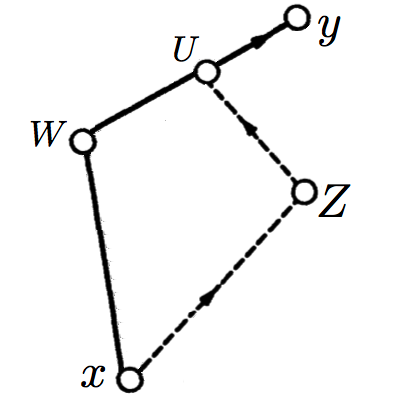} \hspace{3mm}
  \includegraphics[scale=0.35]{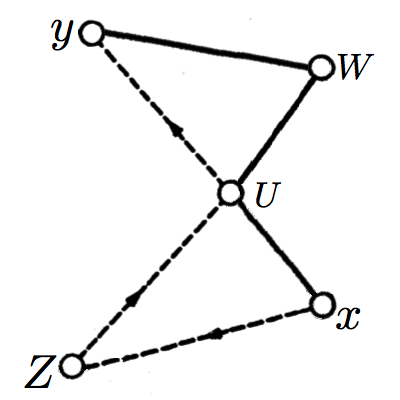} \hspace{3mm}
\includegraphics[scale=0.37]{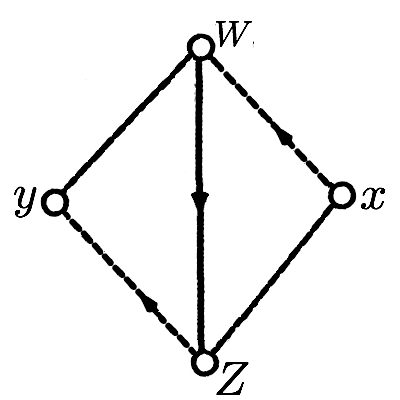}
\par\end{centering}
\caption{\textit{Left:} A pair of paths from $x$ to $y$ in the random connection model which intersect at exactly one vertex $U$, in a specific way indicated by the diagram, i.e. they share only the edge from $U$ to $y$. \textit{Center:} The same event, but occurring in such a way that they share no edges. These motif-like objects are used to calculate the variance of the number of three-hop paths, since one can represent the statistic as the number of ordered \textit{pairs} of paths running between two points. \textit{Right:} Two three-hop paths from $x$ to $y$ intersecting at both their vertices. They share the single edge from $W$ to $Z$.}
\label{fig:paths1}
\end{figure*}
The descending factorials in the two terms in Eq. \ref{e:s1int} are just $(\sigma_{2}\left(\norm{U-x}\right))_{2}$ and $(\sigma_{2}\left(\norm{U-y}\right))_{2}$ respectively, i.e. with $m$ taking both $x$ and $y$, and remembering that with $\Pi \sim \mathrm{Po}(\lambda)$ (i.e. distributed as a Poisson variate) then $\mathbb{E}(\Pi)_{2}=\mathbb{E}\left(\Pi^2\right)-\mathbb{E}\left(\Pi\right)=\lambda^{2}$, then 
\begin{eqnarray}\label{e:squareofexpectation}
\mathbb{E}\left[\left(\sum_{X \in \mathcal{Y}^{\star} \setminus \{U\} }\mathbf{1}\{U \leftrightarrow X\}\mathbf{1}\{X \leftrightarrow m\}\right)_{2}\right] = \left(\mathbb{E}\sigma_2\left(\norm{U-m}\right)\right)^{2}.
\end{eqnarray}
since the term inside the bracket on the left hand side is the number of two hop paths between $U$ and $m$, which is Poisson with expectation $\mathbb{E}\sigma_2\left(\norm{U-m}\right)$.

Eq. \ref{e:squareofexpectation} is simply the number of distinct pairs of two-hop paths from $U$ to $x$ (or $U$ to $y$), so we now have
\begin{multline}\label{e:s1gen}
  \Sigma_{1(1)}=\rho^{3}\int_{\mathbb{R}^{d}}H\left(\norm{x-U}\right) \left(\int_{\mathbb{R}^{d}}H\left(\norm{U-z}\right)H\left(\norm{z-y}\right)\mathrm{d}z\right)^{2}\mathrm{d}U \\ + \rho^{3}\int_{\mathbb{R}^{d}}H\left(\norm{y-U}\right) \left(\int_{\mathbb{R}^{d}}H\left(\norm{U-z}\right)H\left(\norm{z-x}\right)\mathrm{d}z\right)^{2}\mathrm{d}U
\end{multline}
and for the case of Rayleigh fading taking $d,\eta=2$, Eq. \ref{e:s1gen} evaluates to
\begin{eqnarray}\label{e:s1rayleigh}
 \Sigma_{1(1)} = \frac{\pi^3\rho^3}{4\beta^3}\exp{\left( \frac{-\beta\norm{x-y}^2}{2}\right)},
\end{eqnarray}
which appears as the second term in Eq. \ref{e:threevar}.

Now consider $\Sigma_{1(2)}$. This is designed to count pairs of paths which share a single vertex, but in a different way to $\Sigma_{1(1)}$ . This new sort of intersection structure is depicted in the middle panel of Fig. \ref{fig:paths1}. Consider with care, and in relation to this middle panel, the following sum over triples of points:
\begin{equation}\label{e:s12}
\sum_{U \in \mathcal{Y}^{\star}} \sum_{Z \in \mathcal{Y}^{\star} \setminus \{W\} }\mathbf{1}\{x \leftrightarrow Z\}\mathbf{1}\{Z \leftrightarrow U\}\mathbf{1}\{U \leftrightarrow y\}\sum_{W \in \mathcal{Y}^{\star} \setminus \{Z\}}\mathbf{1}\{x \leftrightarrow U\}\mathbf{1}\{U \leftrightarrow W\}\mathbf{1}\{W \leftrightarrow y\}. \nonumber
\end{equation}
This should count the contribution $\Sigma_{1(2)}$.
\begin{figure*}[t]
  \noindent \begin{centering}
     \includegraphics[scale=0.17]{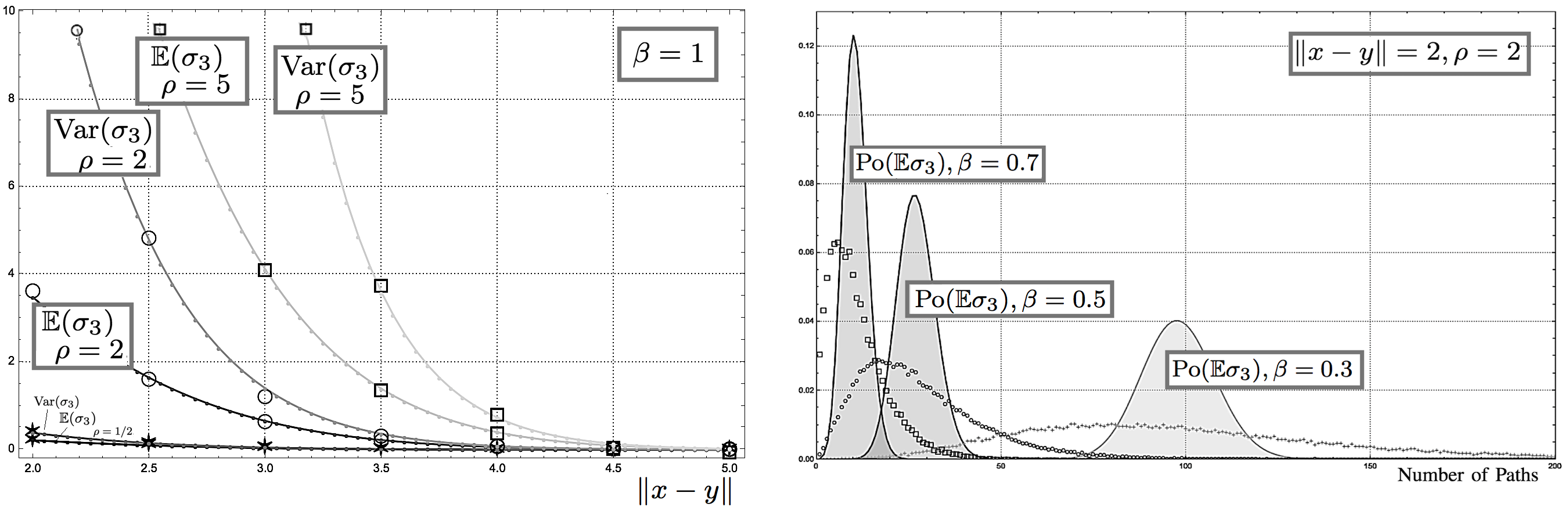}
\par\end{centering}
\caption{\textit{Left}: Numerical corroboration of the analytic mean and variance of $\sigma_{3}$, the number of three-hop paths joining two vertices $x,y$ in $\mathcal{G}_{H}$, for three separate densities $\rho=1/2,2$ and $5$, taking $\beta=1$, scaling over Euclidean separation (horizontal axis). The $\star$, $\circ$ and $\square$'s are Monte Carlo data, averaged over $10^4$ random graphs, whereas the smooth lines are our equations found in Theorems \ref{p:kexpectation} and \ref{p:threehopvariance}.  \textit{Right}: The distribution of the number of $k$-hop paths in the random connection model, fixing $\norm{x-y},\rho=2$, for three values of $\beta=0,7,0.5,0.3$ (left to right). The variance exceeds the mean to a greater extent as the typical connection range grows at fixed density. The Poisson distribution with the analytic mean of Eq. \ref{e:khopexpectation} is plotted for each value of $\beta$ for comparison.}
\label{fig:exp1}
\end{figure*}

In a similar manner to the evaluation of $\Sigma_{1(1)}$, the two inner sums are in fact just counting the number of two hop paths between $x$ and $U$, and also $U$ and $y$, then pairing them with each other, this time including self pairs so there is no descending factorial. The Mecke formula gives the expectation as the following integral:
\begin{multline}\label{e:s12int}
\mathbb{E}\Sigma_{1(1)}=\rho\int_{\mathbb{R}^{d}}H\left(\norm{x-U}\right)H\left(\norm{U-y}\right) \\ \times \mathbb{E}\left[\sum_{Z \in \mathcal{Y}^{\star} \setminus \{W\} }\mathbf{1}\{x \leftrightarrow Z\}\mathbf{1}\{Z \leftrightarrow U\}\sum_{W \in \mathcal{Y}^{\star} \setminus \{Z\}}\mathbf{1}\{U \leftrightarrow W\}\mathbf{1}\{W \leftrightarrow y\}\right]\mathrm{d}U \\ =  \rho\int_{\mathbb{R}^{d}}H\left(\norm{x-U}\right)H\left(\norm{U-y}\right) \\ \times \mathbb{E}\left[\sum_{Z \in \mathcal{Y}^{\star} \setminus \{W\} }\mathbf{1}\{x \leftrightarrow Z\}\mathbf{1}\{Z \leftrightarrow U\}\right]\mathbb{E}\left[\sum_{W \in \mathcal{Y}^{\star} \setminus \{Z\}}\mathbf{1}\{U \leftrightarrow W\}\mathbf{1}\{W \leftrightarrow y\}\right]\mathrm{d}U
\end{multline}
since the two sums are independent. The far right hand side of Eq. \ref{e:s12int} simplifies to
\begin{eqnarray}\label{e:s1int2} \rho\int_{\mathbb{R}^{d}}H\left(\norm{x-U}\right)H\left(\norm{U-y}\right) \mathbb{E}\left(\sigma_{2}\left(\norm{x-U}\right)\right)\mathbb{E}\left(\sigma_{2}\left(\norm{U-y}\right)\right)\mathrm{d}U.
\end{eqnarray}
$\Sigma_{1(2)}$ in terms of a general connection function is therefore 
\begin{multline}\label{e:s12gen}
  \mathbb{E}\Sigma_{1(2)}=\rho^{3}\int_{\mathbb{R}^{d}}H\left(\norm{x-U}\right) H\left(\norm{U-y}\right)\\ \times \left(\int_{\mathbb{R}^{d}}H\left(\norm{x-z}\right)H\left(\norm{z-U}\right)\mathrm{d}z\right) \\ \times \left(\int_{\mathbb{R}^{d}}H\left(\norm{U-z}\right)H\left(\norm{z-y}\right)\mathrm{d}z\right)\mathrm{d}U
\end{multline}
and for the case of Rayleigh fading taking $d,\eta=2$, the third term on the right hand side of Eq. \ref{e:threevar} is
\begin{eqnarray}\label{e:s12rayleigh}
 \mathbb{E}\Sigma_{1(2)} = \frac{1}{6}\exp{\left( \frac{-3\beta\norm{x-y}^2}{4}\right)}
\end{eqnarray}
by integrating the product of exponentials.

There are two more terms, $\Sigma_{2(1)}$ and $\Sigma_{2(2)}$. These both correspond to pairs of paths which share two vertices. Firstly, $\Sigma_{2(1)}$ refers to pairs of paths which share two vertices and all their edges, and so
\begin{eqnarray}\label{e:s21rayleigh}
  \mathbb{E}\Sigma_{2(1)}=\mathbb{E}\sigma_{3}
\end{eqnarray}
since there is a pair of paths for each path, specifically the self-pair. Secondly, $\Sigma_{2(2)}$ refers to pairs which share all their vertices, but not all their edges. This pairing is depicted in the right panel of Fig \ref{fig:paths1}. For this term, we use the following sum, in a similar manner to Eqs. \ref{e:s1int} and \ref{e:s12}
\begin{eqnarray}\label{e:s22}
\Sigma_{2(2)} = \sum_{Z,W \in \mathcal{Y}^{\star}}\mathbf{1}\{x \leftrightarrow Z\}\mathbf{1}\{Z \leftrightarrow W\}\mathbf{1}\{W \leftrightarrow y\}\mathbf{1}\{x \leftrightarrow W\}\mathbf{1}\{W \leftrightarrow Z\}\mathbf{1}\{Z \leftrightarrow y\}
\end{eqnarray}
Now, the shared edge indicator is $\mathbf{1}\{Z \leftrightarrow W\}$, which appears twice,
and so, in a similar manner to the other terms where $W$ and $Z$ are the $d$-dimensional position vectors of the nodes as well as there labels,
\begin{multline}\label{e:s22int}
  \mathbb{E}\Sigma_{2(2)} = \rho^{2}\int_{\mathbb{R}^{d}}H\left(\norm{x-Z}\right)H\left(\norm{Z-W}\right) \\ H\left(\norm{W-y}\right)H\left(\norm{x-W}\right)H\left(\norm{Z-y}\right)\mathrm{d}W\mathrm{d}Z.
\end{multline}
Only once all five links form do these pairs appear, so they are rare, and $\mathbb{E}\Sigma_{2(2)}$ is relatively small. Note that extensive counts of pairs of paths with this property can be an indication of proximity, a point we expand upon in Section \ref{sec:discussion}. Finally, therefore, the last term in Eq. \ref{e:threevar} is
\begin{eqnarray}\label{e:s22rayleigh}
 \Sigma_{2(2)} = \frac{\pi^2\rho^2}{8\beta^2}\exp{\left(-\beta\norm{x-y}^2\right)}
\end{eqnarray}
via evaluating Eq. \ref{e:s22int}, and via Eq. \ref{e:vardecomposition} and then Eqs. \ref{e:s1rayleigh}, \ref{e:s12rayleigh}, \ref{e:s21rayleigh} and \ref{e:s22rayleigh}, the theorem follows.
\end{proof}
We numerically corroborate these formulas via Monte Carlo simulations, the results of which are presented in the top-left panel of Fig. \ref{fig:exp1}.
\section{Discussion}\label{sec:discussion}

\begin{figure}
  \noindent \begin{centering}
     \includegraphics[scale=0.15]{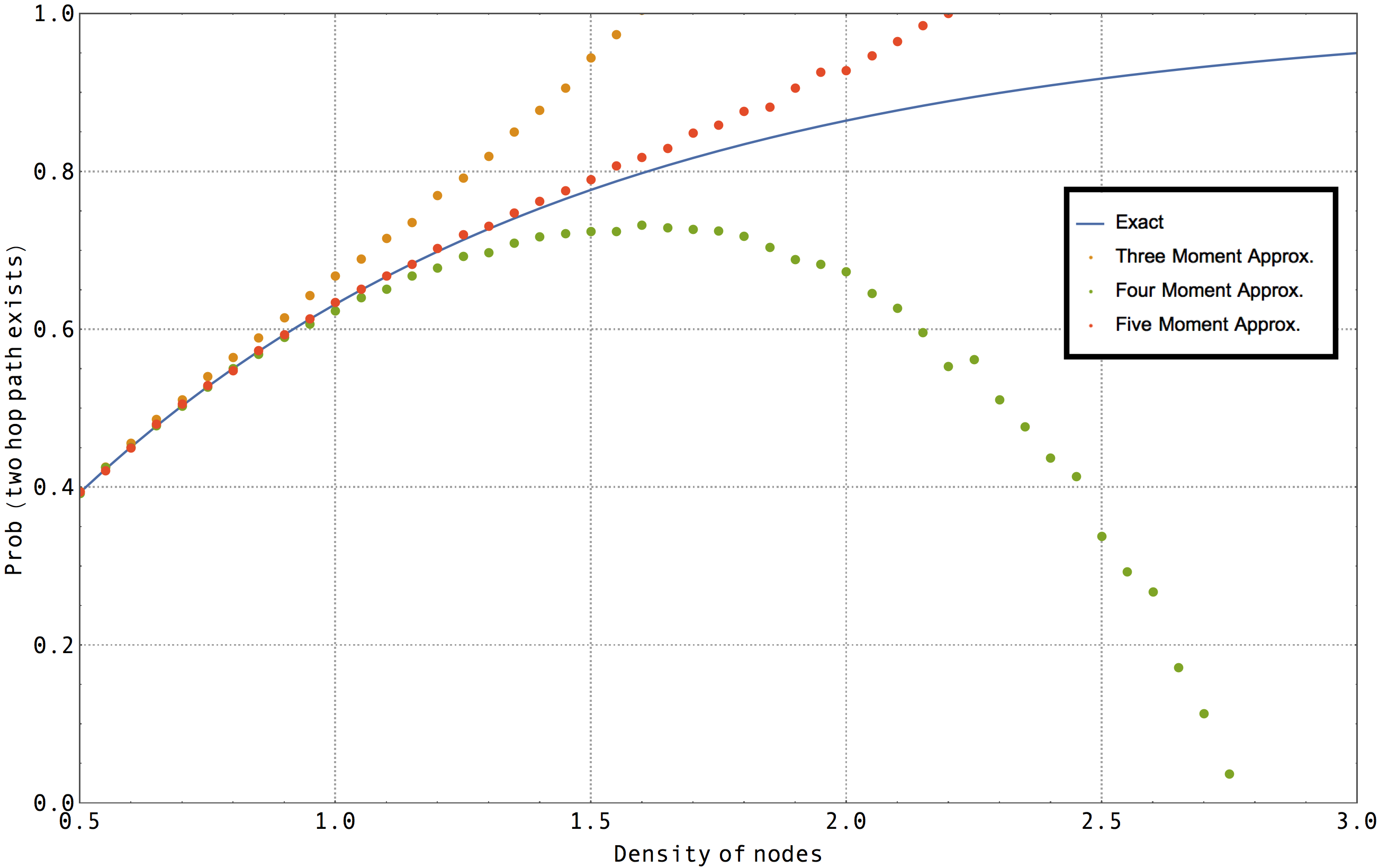} \\
     \includegraphics[scale=0.23]{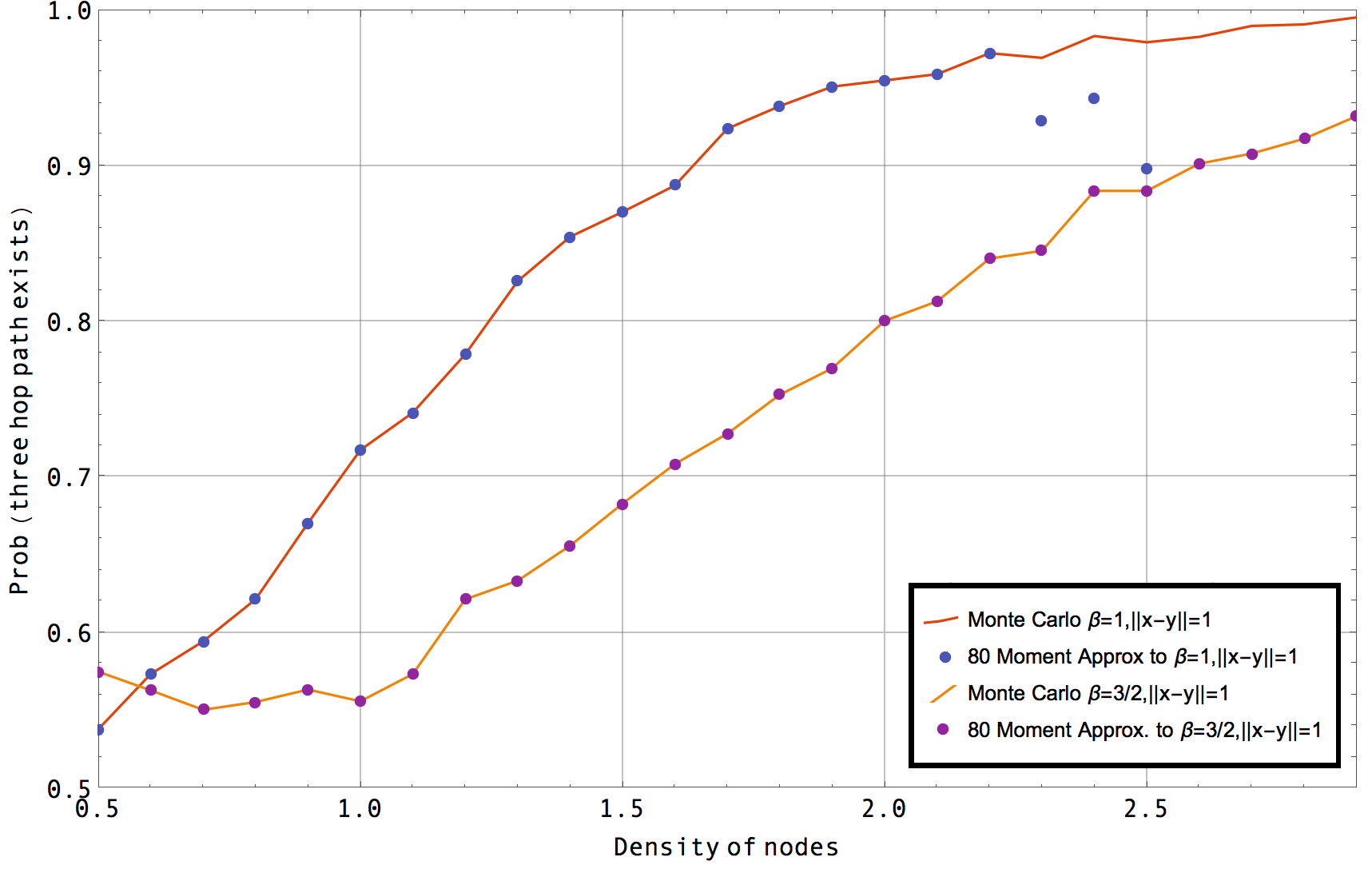}
\par\end{centering}
\caption{\textit{Top}: Both plots take $\beta,\norm{x-y}=1$, though the lower curve in the lower figure takes $\beta=3/2$. We approximate the numerically obtained path existence at various point process densities (solid curves) with sums of factorial moments of $\sigma_{k}$, also numerically obtained (dots), for order $3,4$ and $5$. \textit{Bottom}: Numerically obtained path existence for three hops. We approximate with the sum of 80 factorial moments in each case (dots). We have the analytical points for orders $0,1$ i.e. in terms up to the variance. When the paths are rare, the approximations work better with only a few moments available. These are, as far as we are aware, the most accurate approximations available in the literature.}
\label{fig:moments}
\end{figure}

The main point of discussion here is how to relate the moments of the path count $\sigma_{k}$ to the point probabilities, such as the probability exactly zero paths exist. This is in fact the classic moment problem of mathematical analysis initiated by Thomas J. Stieltjes in $1894$. From the theoretical point of view, the problem has been widely studied and solved many years ago. Various famous mathematicians have contributed, including P. Chebyshev, A. Markov, J. Shohat, M. Frechet, H. Hamburger, M. Riesz, F. Hausdorff and M. Krein \cite{gavriliadis2009}.

We use a relation between the factorial moments of the path count distribution and the point probabilities, as stated in the introduction
\begin{eqnarray}
 P(\sigma_{k}=t) = \frac{1}{t!}\sum_{i \geq 0} \frac{(-1)^{i}}{i!}\mathbb{E}\sigma_{k} (\sigma_{k}-1) \dots (\sigma_{k}-t-i+1)
\end{eqnarray}
we can deduce the probability of a $k$-hop path
\begin{eqnarray}
 P(\sigma_{k}>0) = 1 - \sum_{i \geq 0}\frac{(-1)^{i}}{i!}\mathbb{E}\left[\left(\sigma_{k}\right)_{i}\right]
\end{eqnarray}
where $\left(\sigma_{k}\right)_{i}$ is the descending factorial. The partial sums alternatively upper and lower bound this path-existence probability. We show this in Fig. \ref{fig:moments}, both for the case $k=2$ and the first non-trivial case $k=3$, since the $\sigma_{k}$ is no longer Poisson for all $k>2$. Adding more moments increases the order of the approximation to our desired probability $P(\sigma_{k}>0)$.

We are unable to deduce similar formulas for $k>3$ at this point, but the theory is precisely the same, and does not appear to present any immediate problems. The higher factorial moments, which require knowledge of $\mathbb{E}\sigma_{k}^{3}$ and beyond, are similarly accessible. What is not yet understood is how to form a recursion relation on the moments, which should reveal a ``general theory'' of their form. As such, it may be possible to give as much information as required about $P(\sigma_{k}>0)$ for all $k$. We leave the development of this to a later paper, but highlight that this would give the best possible bounds on the diameter of a random geometric graph in terms of a general connection problem. It remains a key open problem to describe analytically the $r$th factorial moment of $\sigma_{k}$.

\proof[Proof of Theorem \ref{t:moments}]
{For now, and as a key contribution of this article, we can at least present this idea as the proof of a lower bound on the path existence probability for the specific case of three hops, though it holds in general.}

We also have the scaling of the mean and variance of the path count distribution in terms of the node density $\rho$ and the path length $k$.  With $\beta,\norm{x-y}$ fixed, the expected number of paths is $\mathcal{O}(\rho^{k})$, while the variance appears to be $\mathcal{O}(\rho^{k+1})$. We have only verified this for $k=3$. We numerically obtain the probability mass of $\sigma_{3}$ in the top-right panel of Fig. \ref{fig:exp1} by generating $10^{5}$ random graphs and counting all three hop paths between two extra vertices added at fixed at distance $\norm{x-y}$ taking $\rho=2$ for $\beta=0.7,0.5$ and $0.3$. Plotted for comparison is the mass of a Poisson distribution with mean given by Eq. \ref{e:khopexpectation}, which is a spatially independent test case.

It is beyond the scope of this article to analyse the implications this new approximation method will have on e.g. low power localisation in wireless networks, broadcasting, or other areas of application, as discussed in the introduction. We defer this to a later study.

\section{Conclusion}\label{sec:conclusion}
In a random geometric graph known as the random connection model, we derived both the mean and variance of the number of $k$-hop paths between two nodes $x,y$ at displacement $\norm{x-y}$, on condition that $k\in[1,3]$. We also provided details of an example case whenever Rayleigh fading statistics are observed, which is important in applications. This shows how the variance of the number of paths is in fact composed of four terms, no matter what connection function is used. This provides an approximation to the probability that a $k$-hop path exists between distant vertices, and provides technique via summing factorial moments for formulating an accurate approximation in general, which is a sort of correction to a mean field model. This works toward addressing a recent problem of Mao and Anderson \cite{mao2010}. Are results can for example be applied in the industrially important field of connectivity based localisation, where internode distances are estimated without ranging with ulta-wideband sensors, as well as in mathematical problems related to bounding the broadcast time over unknown topologies. 

\section*{Acknowledgements}
This work is supported by Samsung Research Funding and Incubation Center of Samsung Electronics under Project Number SRFC-IT-1601-09. The first author also acknowledges support from the EPSRC Institutional Sponsorship Grant 2015 \textit{Random Walks on Random Geometric Networks}. Both authors wish to thank Mathew Penrose, Kostas Koufos, David Simmons, Leo Laughlin, Georgie Knight, Orestis Georgiou, Carl Dettmann, Justin Coon and Jon Keating for helpful discussions.



\begin{thebibliography}{10}
\providecommand{\url}[1]{#1}
\csname url@samestyle\endcsname
\providecommand{\newblock}{\relax}
\providecommand{\bibinfo}[2]{#2}
\providecommand{\BIBentrySTDinterwordspacing}{\spaceskip=0pt\relax}
\providecommand{\BIBentryALTinterwordstretchfactor}{4}
\providecommand{\BIBentryALTinterwordspacing}{\spaceskip=\fontdimen2\font plus
\BIBentryALTinterwordstretchfactor\fontdimen3\font minus
  \fontdimen4\font\relax}
\providecommand{\BIBforeignlanguage}[2]{{%
\expandafter\ifx\csname l@#1\endcsname\relax
\typeout{** WARNING: IEEEtran.bst: No hyphenation pattern has been}%
\typeout{** loaded for the language `#1'. Using the pattern for}%
\typeout{** the default language instead.}%
\else
\language=\csname l@#1\endcsname
\fi
#2}}
\providecommand{\BIBdecl}{\relax}
\BIBdecl

\bibitem{mao2010}
G.~Mao, Z.~Zhang, and B.~Anderson, ``{Probability of k-Hop Connection under
  Random Connection Model},'' \emph{Communications Letters, IEEE}, vol.~14,
  no.~11, pp. 1023--1025, 2010.

\bibitem{diaz2016}
J.~D\'{i}az, D.~Mitsche, G.~Perarnau, and X.~P\'{e}rez-Gim\'{e}nez, ``On the
  relation between graph distance and euclidean distance in random geometric
  graphs,'' \emph{Advances in Applied Probability}, vol.~48, no.~3, p.
  848–864, 2016.

\bibitem{ta2007}
X.~Ta, G.~Mao, and B.~Anderson, ``{On the Probability of k-hop Connection in
  Wireless Sensor Networks},'' \emph{IEEE Communications Letters}, vol.~11,
  no.~9, 2007.

\bibitem{mao2011}
G.~Mao and B.~D. Anderson, ``{On the Asymptotic Connectivity of Random Networks
  under the Random Connection Model},'' \emph{{INFOCOM, Shanghai, China}}, p.
  631, 2011.

\bibitem{sklar1997}
B.~Sklar, ``Rayleigh fading channels in mobile digital communication systems.
  Part I: Characterization,'' \emph{IEEE Communications Magazine}, vol.~35,
  no.~7, 1997.

\bibitem{giles2016}
A.~P. Giles, O.~Georgiou, and C.~P. Dettmann, ``{Connectivity of Soft Random
  Geometric Graphs over Annuli},'' \emph{The Journal of Statistical Physics},
  vol. 162, no.~4, p. 1068–1083, 2016.

\bibitem{chandler1989}
S.~Chandler, ``{Calculation of Number of Relay Hops Required in Randomly
  Located Radio Network},'' \emph{Electronics Letters}, vol.~25, no.~24, pp.
  1669--1671, 1989.

\bibitem{bettstetter2003}
C.~Bettstetter and J.~Eberspacher, ``Hop distances in homogeneous ad hoc
  networks,'' in \emph{The 57th IEEE Semiannual Vehicular Technology Conference
  (VTC), 2003}, vol.~4, 2003, pp. 2286--2290.

\bibitem{bradonjic2010}
M.~Bradonjić, R.~Elsässer, T.~Friedrich, T.~Sauerwald, and A.~Stauffer,
  \emph{Efficient Broadcast on Random Geometric Graphs}, 2010.

\bibitem{bollobas2001}
B.~Bollobás, \emph{Random Graphs}, 2nd~ed., ser. Cambridge Studies in Advanced
  Mathematics.\hskip 1em plus 0.5em minus 0.4em\relax Cambridge University
  Press, 2001.

\bibitem{penrose2013}
M.~Penrose, ``{Random graphs: The Erd\H{o}s-R\'{e}nyi model},'' in \emph{Talk
  at Oberwolfach Workshop on Stochastic Analysis for Poisson Point Processes,
  February 2013, available at the homepage of the author}, 2013.

\bibitem{iyer2015}
S.~K. Iyer, ``{Connecting the Random Connection Model},'' \emph{preprint
  arXiv:1510.05440}, 2015.

\bibitem{mao2013}
G.~Mao and B.~D.~O. Anderson, ``Connectivity of large wireless networks under a
  general connection model,'' \emph{IEEE Transactions on Information Theory},
  vol.~59, no.~3, pp. 1761--1772, 2013.

\bibitem{penrose2016}
M.~D. Penrose, ``{Connectivity of Soft Random Geometric Graphs},'' \emph{The
  Annals of Applied Probability}, vol.~26, no.~2, pp. 986--1028, 2016.

\bibitem{cef2012}
J.~Coon, C.~Dettmann, and O.~Georgiou, ``{Full Connectivity: Corners, Edges and
  Faces},'' \emph{The Journal of Statistical Physics}, vol. 147, no.~4, pp.
  758--778, 2012.

\bibitem{penrosebook}
M.~D. Penrose, \emph{Random Geometric Graphs}.\hskip 1em plus 0.5em minus
  0.4em\relax Oxford University Press, 2003.

\bibitem{penrose1997}
------, ``{The Longest Edge of the Random Minimal Spanning Tree},'' \emph{The
  Annals of Applied Probability}, vol.~7, no.~2, pp. 340--361, 1997.

\bibitem{clifford1990}
P.~Clifford, ``{Markov Random Fields in Statistics},'' in \emph{{Disorder in
  Physical Systems: A Volume in Honour of John M. Hammersley}}, {Geoffrey
  Grimmett and Dominic Welsh}, Ed.\hskip 1em plus 0.5em minus 0.4em\relax
  Cambridge University Press, 1990.

\bibitem{goulden2004}
I.~P. Goulden and D.~M. Jackson, \emph{Combinatorial Enumeration}.\hskip 1em
  plus 0.5em minus 0.4em\relax Dover Publications, 2004.

\bibitem{giles2017}
A.~P. Kartun-Giles, G.~Knight, O.~Georgiou, and C.~P. Dettmann, ``The
  electrical resistance of a one-dimensional random geometric graph,'' \emph{In
  preparation}, 2017.

\bibitem{knight2017}
G.~Knight, A.~P. Kartun-Giles, O.~Georgiou, and C.~P. Dettmann, ``Counting
  geodesic paths in 1-d vanets,'' \emph{IEEE Wireless Communications Letters},
  vol.~6, no.~1, pp. 110--113, 2017.

\bibitem{dulman2006}
S.~Dulman, M.~Rossi, P.~Havinga, and M.~Zorzi, ``On the hop count statistics
  for randomly deployed wireless sensor networks,'' \emph{Int. J. Sen. Netw.},
  vol.~1, no. 1/2, pp. 89--102, Sep. 2006.

\bibitem{ge2016}
X.~Ge, S.~Tu, G.~Mao, C.~X. Wang, and T.~Han, ``{5G Ultra-Dense Cellular
  Networks},'' \emph{IEEE Wireless Communications}, vol.~23, no.~1, pp. 72--79,
  2016.

\bibitem{ng2010}
S.~C. Ng and G.~Mao, ``Analysis of k-hop connectivity probability in 2-d
  wireless networks with infrastructure support,'' in \emph{2010 IEEE Global
  Telecommunications Conference GLOBECOM 2010}, 2010, pp. 1--5.

\bibitem{zemlianov2005}
A.~Zemlianov and G.~de~Veciana, ``Capacity of ad hoc wireless networks with
  infrastructure support,'' \emph{IEEE Journal on Selected Areas in
  Communications}, vol.~23, no.~3, pp. 657--667, 2005.

\bibitem{perevalov2005}
E.~Perevalov, R.~Blum, A.~Nigara, and X.~Chen, ``Route discovery and capacity
  of ad hoc networks,'' in \emph{GLOBECOM '05. IEEE Global Telecommunications
  Conference, 2005.}, vol.~5, 2005, pp. 6 pp.--2740.

\bibitem{nguyen2015}
C.~Nguyen, O.~Georgiou, and Y.~Doi, ``{Maximum Likelihood Based Multihop
  Localization in Wireless Sensor Networks},'' in \emph{Proceedings. IEEE ICC,
  London, UK}, 2015.

\bibitem{elsasser2007}
R.~Els{\"a}sser and T.~Sauerwald, \emph{Broadcasting vs. Mixing and Information
  Dissemination on Cayley Graphs}, 2007, pp. 163--174.

\bibitem{peleg2007}
D.~Peleg, ``Time-efficient broadcasting in radio networks: A review,'' in
  \emph{International Conference on Distributed Computing and Internet
  Technology}, 2007, pp. 1--18.

\bibitem{jagers1973}
P.~Jagers, ``On palm probabilities,'' \emph{Zeitschrift f{\"u}r
  Wahrscheinlichkeitstheorie und Verwandte Gebiete}, vol.~26, no.~1, pp.
  17--32, 1973.

\bibitem{penrose20162}
M.~D. Penrose, ``Lectures on random geometric graphs,'' in \emph{Random Graphs,
  Geometry and Asymptotic Structure}, N.~Fountoulakis and D.~Hefetz, Eds.\hskip
  1em plus 0.5em minus 0.4em\relax Cambridge University Press, 2016, pp.
  67--101.

\bibitem{haenggi2009}
M.~Haenggi, J.~Andrews, F.~Baccelli, O.~Dousse, and M.~Franceschetti,
  ``Stochastic geometry and random graphs for the analysis and design of
  wireless networks,'' \emph{Selected Areas in Communications, IEEE Journal
  on}, vol.~7, pp. 1029--1046, 2009.

\bibitem{elsawy2017}
H.~ElSawy, A.~Sultan-Salem, M.~S. Alouini, and M.~Z. Win, ``Modeling and
  analysis of cellular networks using stochastic geometry: A tutorial,''
  \emph{IEEE Communications Surveys Tutorials}, vol.~19, no.~1, pp. 167--203,
  2017.

\bibitem{gavriliadis2009}
\BIBentryALTinterwordspacing
P.~Gavriliadis and G.~Athanassoulis, ``Moment information for probability
  distributions, without solving the moment problem, ii: Main-mass, tails and
  shape approximation,'' \emph{Journal of Computational and Applied
  Mathematics}, vol. 229, no.~1, pp. 7 -- 15, 2009. [Online]. Available:
  \url{http://www.sciencedirect.com/science/article/pii/S037704270800513X}
\BIBentrySTDinterwordspacing

\end{thebibliography}
\end{document}